\documentclass[a4paper,12pt]{amsart}
\usepackage{amsmath,amssymb}
\usepackage[left=2cm,right=2cm]{geometry}
\def\C{\Bbb C}

\def\hol{\mathcal O}

\def\O{\Omega}

\def\k{\kappa}

\def\vp{\varphi}
\def\ov{\overline}

\def\lk{ l^{\kappa}}
\def\Cn{\mathbb{C}^n}

\newtheorem{thm}{Theorem}
\newtheorem{prop}[thm]{Proposition}
\newtheorem{cor}[thm]{Corollary}
\newtheorem{rem}[thm]{Remark}
\newtheorem{defn}[thm]{Definition}

\title[Strong localizations of the Kobayashi Distance]
{Strong Localizations of the Kobayashi Distance}

\author{Nikolai Nikolov and Ahmed Yekta Ökten}
\address{N. Nikolov\\
	Institute of Mathematics and Informatics\\
	Bulgarian Academy of Sciences\\
	Acad. G. Bonchev Str., Block 8\\
	1113 Sofia, Bulgaria}

\address{Faculty of Information Sciences\\
	State University of Library Studies
	and Information Technologies\\
	69A, Shipchenski prohod Str.\\
	1574 Sofia, Bulgaria}

\email{nik@math.bas.bg}

\address{A. Y. Ökten\\
	Institut de Math\'ematiques de Toulouse; UMR5219 \\
	Universit\'e de Toulouse; CNRS \\
	UPS, F-31062 Toulouse Cedex 9, France} \email{ahmed$\_$yekta.okten@math.univ-toulouse.fr}

\thanks{ The first author was partially supported by the
	Bulgarian National Science Fund, Ministry of Education and Science of Bulgaria
	under contract KP-06-N52/3.	
	The second author received support from the University Research School EUR-MINT
	(State support managed by the National Research Agency for Future Investments
	program bearing the reference ANR-18-EURE-0023). }
	
\subjclass[2010]{32F45}

\begin{document}

\keywords{Kobayashi distance, Kobayashi-Royden metric, localization}

\begin{abstract} Recently, the visibility property of Kobayashi (almost) geodesics has been used to provide localizations of the Kobayashi distance. In this note, we provide sufficient growth conditions for the Kobayashi distance to obtain new strong multiplicative and additive localization results. Curiously, the conditions we provide are deeply related to the behaviour of the Kobayashi geodesics.
\end{abstract}

\maketitle

\section{Introduction}

The Kobayashi distance is a celebrated tool in complex analysis, therefore it is often of interest to study its behaviour. In many cases one may observe localizations of the Kobayashi distance, that is to say, in many occasions the Kobayashi distance of a domain and its subdomains are comparable.

In particular, the visibility property of Kobayashi geodesics (with respect to Euclidean boundary) is first exploited in \cite{BNT} to give an additive localization result in the case of convexity. Thanks to some important observations, Sarkar \cite{S} removed the convexity assumption from the aforementioned result. To localize the notion of visibility itself, \cite{NOT} presented a local version of the additive localization result given in \cite{S}. The papers \cite{S,NOT} also provide multiplicative localization results in the same spirit.

In this article, we introduce the notion of $v$-points and $w$-points in terms of metric properties of the Kobayashi distance arising from the behaviour of geodesics. Our goal is to present strong multiplicative and additive localization results near such points, namely Theorem \ref{multiplicativelocalizationresult} and Theorem \ref{additivelocalizationresult}. More explicitly, under our assumptions,  the ratio of the Kobayashi distances in a domain and in its suitable subdomains is arbitrarily close to one and their difference is arbitrarily close to zero. We also conclude that due to \cite[Section 6]{BNT}, our results cover various interesting cases including strong pseudoconvexity and more generally  $\C$-strictly convexifiability.

\section{Results}\label{secresults}
Before we state our results, let us recall some terminology.

Unless stated otherwise, in what follows, $\O$ denotes a domain in $\Cn$.

Let $z,w\in \O$. The Kobayashi distance $k_\Omega$ is the largest pseudodistance
not exceeding the Lempert function $ l_\O(z,w):=\tanh^{-1} \tilde{l}_\O(z,w),$
where $\Delta$ is the
unit disc and $$\tilde{l}_\Omega(z,w):=\inf_{\alpha\in \Delta}\{|\alpha|:\exists\vp\in\hol(\Delta,\O)
\hbox{ with }\varphi(0)=z,\varphi(\alpha)=w\} .$$ %$\tanh^{-1}t=\frac12\log\frac{1+t}{1-t}$.

The Kobayashi-Royden pseudometric is defined as
$$\k_\O(z;v)=\inf_{\lambda\in \C}\{|\lambda|:\exists\vp\in\O(\Delta,\O)\hbox{ with }
\varphi(0)=z,\lambda\varphi'(0)=v\},$$
and the Kobayashi-Royden length of an absolutely continuous curve $\gamma:I\rightarrow \O$ is given as $$ \lk_{\O}(\gamma):=\int_{I} \k_\O(\gamma(t);\gamma'(t))dt. $$

The Kobayashi pseudodistance is the integrated form of the Kobayashi-Royden pseudometric (see \cite[Theorem 1.3]{V}), that is
$$
	k_\O(z,w)=\inf\lk_{\O}(\gamma) ,$$ where the infimum is taken over all absolutely continuous curves joining $z$ to $w$.

Recall that $\O$ is said to be \emph{hyperbolic} if $k_\O$ is a distance. When studying local properties, global hyperbolicity is often not necessary. This leads us to say that $\O$ is \emph{hyperbolic near} $p\in\partial \O$ if there exists a neighbourhood $U$ of $p$ such that
$z,w\in\O\cap U$, $z\neq w$ implies \begin{equation}\label{hyperbolicitynearp}
	k_\O(z,w)>0.
\end{equation}

Another related notion is hyperbolicity at $p\in\partial \O$. Recall the definition given in \cite{NOT}. If $\O$ is bounded $\O$ is
\emph{hyperbolic at}  $p$ for any $p\in \partial \O$. If $\O$ is unbounded, then $\O$ is \emph{hyperbolic at} $p\in\partial \O$ if we have
\begin{equation}
	\label{localhyperbolicity}
	\liminf_{z\to p,w\to \infty} l_\O(z,w) > 0. \end{equation}

For $A,B\subset \O$ denote $ k_\O(A,B):=\inf_{a\in A, b\in B} k_\O(a,b)$.
If $\O$ is an unbounded domain, by \cite[Proposition 2]{NOT}, \eqref{localhyperbolicity} holds if and only if there exists a bounded neighbourhood $U$ of $p$, such that for any  other two neighbourhoods of $p$ satisfying $V'\subset\subset U'\subset U$ we have
\begin{equation}\label{strongerlocalhyperbolicity}
	k_\O(\O\cap V', \O\setminus U') > 0.
\end{equation}

It is worth noting that \eqref{hyperbolicitynearp} and \eqref{strongerlocalhyperbolicity} hold for any bounded domain.

Let $\epsilon>0$. In general, geodesics for the Kobayashi distance need not to exist. This motivates us to recall that an absolutely continuous curve $\gamma:I\rightarrow \O$ joining $z$ to $w$ is said to be an \emph{$\epsilon$-geodesic} if we have
\begin{equation}\label{epsilongeodesic}
	\lk_{\O}(\gamma) \leq k_\O(z,w) + \epsilon.
\end{equation}
The existence of $\epsilon$-geodesics for any $\epsilon>0$ is an immediate consequence of the fact that the Kobayashi distance is the integrated form of the Kobayashi-Royden pseudometric.

Recall now the definition of Gromov product:
$$ (z|w)^{\O}_o=\dfrac{1}{2}\left( k_\O(z,o)+k_\O(w,o)-k_\O(z,w)\right),\quad  z,w,o\in \O.$$

We say that $\O$ satisfies the \emph{Gromov property} at $p\in\partial \O$ (\cite[Definition 5]{NOT})
if for some (hence all) $o\in\O$ we have
\begin{equation}\label{gromovproperty}
	\limsup_{z\rightarrow p, w\rightarrow q} (z|w)^\O_o < \infty,
\end{equation} and $\O$ satisfies the \emph{weak Gromov property} at $p\in\partial \O$ if \begin{equation}\label{weakgromovproperty}
	\liminf_{z\rightarrow p, w\rightarrow q} \left(k_\O(z,w)-k_\O(z,o)\right) > -\infty \: \: \: \forall o \in \O, \: \forall q\neq p .
\end{equation}

Obviously, if $\O$ satisfies the Gromov property at $p,$ then it satisfies the weak Gromov property at $p$.

Further, note that geodesics for the Kobayashi distance might not exist in general. On the other, it is clear
that any $\epsilon>0$ there exists an \emph{$\epsilon$-geodesic}, that is
an absolutely continuous curve $\gamma:I\rightarrow \O$ joining $z$ to $w$ such that
\begin{equation}\label{epsilongeodesic}
	\lk_{\O}(\gamma) \leq k_\O(z,w) + \epsilon.
\end{equation}

Let us recall that $p\in\partial \O$ (see \cite[Definition 4]{NOT}) is said to be a \textit{weakly visible} point if for any bounded neighbourhood $U$ of $p$, there exists $V\subset\subset U$ such that if $\gamma:I\rightarrow \O$ is an $\epsilon$-geodesic joining a point in $\O\cap V$ to a point in $\O\setminus U$ then $\gamma(I)$ intersects a compact set (depending on $\epsilon$) $K_\epsilon\subset\subset\O$. As a consequence of \cite[Proposition 2.5]{BNT} one can see that satisfying the Gromov property is equivalent to being a weakly visible point.

\begin{defn}\label{newpointtypes} We say that $p\in\partial \O$ is a \emph{$v$-point} if $\O$ satisfies the weak Gromov property at $p$.

 We say that $p\in\partial \O$ is a \emph{$w$-point} if for any neighbourhood $U$ of $p$ such that $\O\setminus U$ is non-empty we have
		\begin{equation}\label{Wpointcondition}
			\lim_{z,w\rightarrow p}\left(\inf_{o\in \O\setminus U}(z|w)^\O_o\right)=\infty.
		\end{equation}
	\end{defn}

Notice that \eqref{Wpointcondition} is stronger than the following notion introduced in \cite{N}: \begin{equation}\label{weakWpointcondition} \liminf_{z,w\rightarrow p} (z|w)^\O_o = \infty \: \: \: \: \: \forall o \in \O.
\end{equation}

It is easy to see that under the assumption of \eqref{weakgromovproperty}, \eqref{Wpointcondition} and \eqref{weakWpointcondition} are equivalent.

We extend the notion of well-behaved geodesics given in \cite{BZ}, saying that $p\in\partial\O$ has \emph{well-behaved $\epsilon$-geodesics} if for some (hence any) $o\in\O$, any $z_n,w_n\rightarrow p$ and any $\epsilon$-geodesics $\gamma_n:I_n\rightarrow \O$ joining $z_n$ to $w_n$ we have $k_\O(\gamma_n(I_n),o)\rightarrow \infty$.

\begin{prop}\label{prop:wellbehavedandwpt} The following are equivalent for $p\in\partial\O:$

	(a) $p$ has well-behaved $\epsilon$-geodesics for any $\epsilon>0$;

    (b) \eqref{weakWpointcondition} holds.
\end{prop}

We say that $\O$ satisfies weak multiplicative, resp. additive localization at $p\in\partial\O$ if for any neighbourhood $U$ of $p$ one has that
$$\limsup_{z,w\rightarrow p}\dfrac{k_{\O\cap U}(z,w)}{k_\O(z,w)}<\infty,$$
$$\mbox{resp. }\limsup_{z,w\rightarrow p} \left(k_{\O\cap U}(z,w)-k_{\O}(z,w)\right)<\infty.$$
Here and below we assume that $z,w$ belong to the same connected component of $\O\cap U$. We would like to mention the weak multiplicative localization
results in \cite[Theorem 14]{NOT} and \cite[Theorem 1.6]{S}, and the weak additive localization results in
\cite[Theorem 1.4]{BNT}, \cite[Theorem 9]{NOT}, and \cite[Theorem 1.3]{S}.

We present our main results which can be considered as \textit{strong} multiplicative, resp. additive localization. They extend, for instance, \cite[Proposition 3]{V1} and \cite[Proposition 6.5]{BNT}.

\begin{thm}\label{multiplicativelocalizationresult}
	Let $\O$ be a domain in $\Cn$ and $p\in\partial\O$ be a $w$-point. Then for any neighbourhood $U$ of $p$ we have \begin{equation}\label{multiplicativelocalizationequation} \lim_{z,w\rightarrow p}\dfrac{k_{\O\cap U}(z,w)}{k_\O(z,w)}=1 .
	\end{equation}
\end{thm}

Note that Proposition \ref{Wpointsandhyperbolicity} below implies that under the hypothesis of Theorem \ref{multiplicativelocalizationresult},
$k_\O(z,w)>0$ for $z,w\in\O$ near $p.$ Therefore the limit in \eqref{multiplicativelocalizationequation} is well defined.

\begin{thm}\label{additivelocalizationresult}
	Let $\O$ be a domain in $\Cn$, $p\in\partial \O$ be a $v$-point and a $w$-point. Then for any neighbourhood $U$ of $p$, we have
\begin{equation}\label{additivelocalizationequation}
	\lim_{z,w\rightarrow p} \left(k_{\O\cap U}(z,w)-k_{\O}(z,w)\right)=0.
\end{equation}
\end{thm}

We note that our results are indeed "strong" localizations because the constants given in limits \eqref{multiplicativelocalizationequation} and \eqref{additivelocalizationequation} are the best constants one can get.

We defer the proofs of Theorems \ref{multiplicativelocalizationresult} and \ref{additivelocalizationresult} to Section \ref{secproofs}.

\begin{rem}
	 The assumptions of Theorems \ref{multiplicativelocalizationresult} and \ref{additivelocalizationresult} may be replaced with conditions depending on the behaviour of ($\epsilon$-)geodesics, which are reminiscents of hyperbolic geometry.
\end{rem}	

To be more explicit, observe that our previous discussion and Proposition \ref{prop:wellbehavedandwpt} imply that in Theorem \ref{additivelocalizationresult} we may replace $p\in\partial\O$ being a $v$-point with being a weakly visible point and $p$ being a $w$-point with $p$ having well-behaved $\epsilon$-geodesics for any $\epsilon>0$. On the other hand, in Theorem \ref{multiplicativelocalizationresult} in order to replace $p\in\partial\O$ being a $w$-point with $p$ having well-behaved $\epsilon$-geodesics for any $\epsilon>0$, we need to assume that \eqref{weakgromovproperty} holds for the pair $\{\O,p\}$, that is $p\in\partial\O$ is a $v$-point.

As they are seemingly independent of ($\epsilon$-)geodesics, we prefer to state our theorems under the $v$-point and $w$-point conditions.

To get corollaries of Theorems \ref{multiplicativelocalizationresult} and \ref{additivelocalizationresult}, we continue the discussion.

Let $\delta_\O(x):=\inf\{\|x-y\|:y\in\partial\O\}$.
Recall \cite[Definition 6.1]{BNT}. $p\in\partial \O$ is said to be a \emph{$k$-point} (of $\O$) if for any neighbourhood $U$ of $p$ one has that
\begin{equation}\label{kpoint}
	\liminf_{z\rightarrow p} \left( k_\O(z,\O\setminus U)-\dfrac{1}{2}\log\left(\dfrac{1}{\delta_\O(z)}\right)\right) >-\infty .
\end{equation}

Dini-smoothness is a property stronger than $\mathcal{C}^{1}$ smoothness and weaker than $\mathcal{C}^{1,\epsilon}$ smoothness.
By \cite[Theorem 7]{NA}, if $p$ is a Dini-smooth boundary point if a domain $\O\subset\Cn,$ then (see \cite[Theorem 7]{NA})
\begin{equation}\label{dinismoothestimate}
	k_\O(z,w)\leq \log\left(1+\dfrac{2\|z-w\|}{\sqrt{\delta_\O(z)\delta_\O(w)}}\right), \:\:\:\:\: z,w\in\O\mbox{ near }p.
\end{equation}

Note also that the proof of \cite[Proposition 6.15]{BNT} implies that:

\begin{prop}\label{newversionofBNTresult}
	Let $\O$ be a domain in $\Cn$ (not necessarily bounded) and assume that $p\in\partial \O$ is a $k$-point and $\partial \O$ is Dini-smooth near $p$. Then for any neighbourhood $U$ of $p$ we have
	$$ \lim_{z,w\rightarrow p} \left(k_{\O\cap U}(z,w)-k_\O(z,w)\right)=0.$$
\end{prop}

We claim the following:

\begin{prop}\label{kpointsareVWpoints}
	Let $p\in\partial \O$ be a $k$-point and assume that $\partial \O$ is Dini-smooth in a neighbourhood of $p$. Then $p$ is a $v$-point and a $w$-point.
\end{prop}

We also defer the proof of Proposition \ref{kpointsareVWpoints} to Section \ref{secproofs}.

By Proposition \ref{kpointsareVWpoints}, Theorem \ref{additivelocalizationresult} generalizes Proposition \ref{newversionofBNTresult}. Moreover, by Theorem \ref{multiplicativelocalizationresult} and Proposition \ref{kpointsareVWpoints} we get:

\begin{cor}\label{kpointdinismoothmultiplicativelocalization}
	Let $p\in\partial \O$ be a $k$-point and assume that $\partial \O$ is Dini-smooth in a neighbourhood of $p$. Then for any neighbourhood $U$ of $p$, the Kobayashi distance satisfies
	$$ \lim_{z,w\rightarrow p} \dfrac{k_{\O\cap U}(z,w)}{k_\O(z,w)}=1.$$
\end{cor}

Let $\O\subset\Cn$ be a convex domain, $p\in\partial \O$ is said to be a $\mathbb{C}$-strictly convex boundary point if $\partial\O\cap T^{\C}_p \partial \O = \{p\}$. Convexity leads to more precise estimates of the Kobayashi distance. In particular, it is not hard to see that any $\C$-strictly convex boundary point of a convex domain is a $k$-point. However, there is more to say on that. We say $p\in\partial \O$ is a $\C$-strictly convexifiable boundary point if there exists a neighbourhood $U$ of $p$ and a holomorphic embedding $\Psi:U\rightarrow \Cn$ such that $\Psi(\O\cap U)$ is a convex domain and $\Psi(p)$ is a $\C$-strictly convex boundary point.

\begin{prop}\label{cstrictlyimpliesk}\cite[Theorem 6.13]{BNT}
	Let $p\in\partial\O$ be a $\C$-strictly convexifiable boundary point. Then $p$ is a $k$-point.
\end{prop}

As a corollary of Theorem \ref{multiplicativelocalizationresult}, Theorem \ref{additivelocalizationresult}, Proposition \ref{kpointsareVWpoints} and Proposition \ref{cstrictlyimpliesk} we obtain:

\begin{cor}\label{cstrictlyconvexlocalization}
	Let $p\in\partial\O$ be a $\C$-strictly convexifiable boundary point and assume that $\partial\O$ is Dini-smooth near $p$. Then for any neighbourhood $U$ of $p$ we have:
$$\lim_{z,w\rightarrow p}\dfrac{k_{\O\cap U}(z,w)}{k_\O(z,w)}=1,$$	
$$\lim_{z,w\rightarrow p} \left(k_{\O\cap U}(z,w)-k_\O(z,w)\right)=0.$$
\end{cor}

It is well-known that $\mathcal{C}^2$-smooth strongly pseudoconvex boundary points are $\C$-strictly convex boundary points. In particular, Corollary \ref{cstrictlyconvexlocalization} applies to such points. This covers \cite[Proposition 3]{V1}.

Moreover, under global assumptions such as strong pseudoconvexity or convexity and visibility, one may find quantitative strong localization results in \cite{N}. The reader may also consult to \cite{NT} for weaker versions of these results in the case of $\mathbb{C}$-convexifiable finite type boundary points. These results are applied there in order to study the boundary behaviour of the quotient/difference of the Lempert function $l_\O$ and the so-called Carathéodory distance $c_\O$ (recall that $c_\O\leq k_\O \leq l_\O$). Our results can be also applied to such a study.

\section{Proofs}\label{secproofs}

 \textit{Proof of Proposition \ref{prop:wellbehavedandwpt}.}
 (a) $\Rightarrow$ (b). Suppose that \eqref{weakWpointcondition} fails. Then there exist $C>0$ and sequences $z_n,w_n\in\O$ tending to $p\in\partial\O$ with $(z_n|w_n)_o \leq C$. We may take $\epsilon$-geodesics (defined on disjoint intervals) $\sigma^z_n: I^z_n\rightarrow \O,\sigma^w_n:I^w_n\rightarrow\O$ joining $z_n$ to $o$, $w_n$ to $o$ respectively. Let us denote the union of these curves by $\sigma_n$. Explicitly, we set $\sigma_n:I^w_n\cup I^w_n\rightarrow\O$, $\sigma_n(t)=\sigma^z_n(t)$ if $t\in I^z_n$, and $\sigma_n(t)=\sigma^w_n(t)$ otherwise. We claim that $\sigma_n$ are $\epsilon'$-geodesics joining $z_n$ to $w_n$ for some $\epsilon'>0$. To see this note that  $(z_n|w_n)_o \leq C$ implies that $k_\O(z_n,o)+k_\O(w_n,o)\leq k_\O(z_n,w_n) + 2C.$
	Since $\sigma^z_n,\sigma^w_n$ are $\epsilon$-geodesics, $l_\k^\O(\sigma^z_n)\leq  k_\O(z_n,o) + \epsilon$ and  $l_\k^\O(\sigma^w_n)\leq  k_\O(w_n,o) + \epsilon$. Consequently, $$l_\k^\O(\sigma_n)\leq l_\k^\O(\sigma^z_n)+l_\k^\O(\sigma^w_n) \leq   k_\O(z_n,o) + \epsilon +   k_\O(w_n,o) + \epsilon \leq k_\O(z_n,w_n)+ 2C + 2 \epsilon.$$
	Setting $\epsilon'=2C+2\epsilon$ we see that $\sigma_n$ are  $\epsilon'$-geodesics whose image contains $o\in\O$. This contradicts the fact that $p$ has well-behaved $\epsilon'$-geodesics.
	
(b) $\Rightarrow$ (a). Suppose that $p$ does not have well-behaved $\epsilon$-geodesics for some $\epsilon>0$. Then we may take $z_n,w_n$ tending to $p$ and $\epsilon$-geodesics $\gamma_n:I_n\rightarrow\O$ intersecting a compact set $K\subset\subset\O$. Take $s_n \in \gamma_n(I_n)\cap K$. Clearly $k_\O(s_n,o)\leq C' < \infty$, where $C'$ is the diameter of the set $K$ with respect to $k_\O$. Then
	$$ k_\O(z_n,o)+k_\O(w_n,o) \leq k_\O(s_n,z_n)+ k_\O(s_n,w_n) + 2k_\O(s_n,o) \leq l_\kappa(\gamma_n)+ 2C' \leq k_\O(z_n,w_n) + 2C' + \epsilon. $$
	Then we see that $(z_n|w_n)_o \leq C'+\epsilon/2$ so \eqref{weakWpointcondition} fails.\qed
\smallskip

    \textit{Proof of Proposition \ref{kpointsareVWpoints}.} We will first show that $p$ is a $v$-point.
	
	To do so, note that since $\partial \O$ is Dini-smooth in a neighbourhood of $p$, we can choose a small enough neighbourhood $U$ of $p$ so that \eqref{dinismoothestimate} holds whenever $z,w\in \O\cap U$.
	
	Now we suppose the contrary and assume that $\eqref{weakgromovproperty}$ fails. Then there exists $o\in\O$ and $ q \in \partial \O, q\neq p$, such that we can find sequences of points $z_n\to p$, $w_n\to q$
	with $$k_\O(z_n,o)\geq k_\O(z_n,w_n) +n. $$
	
	As $ o\in \O$, there exists $C>0$ and $o'\in\O\cap U$ satisfying $ k_\O(z_n,o)-k_\O(z_n,o') \leq k_\O(o',o) \leq C $ for all $n\in \mathbb{N}$. In particular, this implies that $k_\O(z_n, o')\geq k_\O(z_n,w_n) + n - C$. As $o'\in\O\cap U$ and $\{z_n\}_{n\in\mathbb{N}}\subset\O\cap U$ by \eqref{dinismoothestimate} we have \begin{equation}\label{upperboundforkpoints}\k_\O(z_n,o')\leq \log\left(1+\dfrac{2\|z_n-o'\|}{\sqrt{\delta_\O(z_n)\delta_\O(o')}}\right).
	\end{equation}
	By above we obtain $$k_\O(z_n,w_n)\leq \dfrac{1}{2}\log\left(\dfrac{1}{\delta_\O(z_n)}\right)+C'-n.$$ As $z_n\rightarrow p$ and $w_n\rightarrow q\neq p$ this contradicts with the fact that $p$ is a $k$-point. Hence $p$ must be a $v$-point.
	
	Now, we will show that $p$ is a $w$-point.
	
	Let $U$ be as above and assume that $p$ is not a $w$-point. Then we can find $C>0$, $z_n,w_n\rightarrow p$, $o_n\rightarrow o\in\ov \O $
	with $o\neq p$ such that
	\begin{equation}\label{kpointimplieswpointeqn} k_\O(z_n,o_n)+k_\O(w_n,o_n)-k_\O(z_n,w_n) \leq C.
	\end{equation}
	As $z_n,w_n\rightarrow p$, by passing to the limit we may assume that $\{z_n\}_{n\in \mathbb{N}},\{w_n\}_{n\in\mathbb{N}}\subset \O\cap U$ so by \eqref{kpoint} we obtain $$ k_\O(z_n,w_n)\leq \log\left(1+\dfrac{2\|z_n-w_n\|}{\sqrt{\delta_\O(z_n)\delta_\O(w_n)}}\right) .$$
	Hence by \eqref{kpointimplieswpointeqn} we get
	$$ k_\O(z_n,o_n)+k_\O(w_n,o_n) \leq \log\left(1+\dfrac{2\|z_n-w_n\|}{\sqrt{\delta_\O(z_n)\delta_\O(w_n)}}\right) + C .$$
	$z_n,w_n\rightarrow p$ so $\|z_n-w_n\|\rightarrow 0$. As $o_n\rightarrow o\neq p$, this contradicts the fact that $p$ is a $k$-point. Hence
$p$ must be a $w$-point.\qed
\smallskip

As in \cite{BNT,S,NOT}, the key estimate in the proofs of our localization results will be Royden's localization lemma. Recall that (see for instance \cite[Proposition 13.2.10]{JP}) if $\O$ is a domain in $\Cn$ and $D$ is any subdomain, we have
\begin{equation}\label{roydenslocalizationlemma}
	\tilde{l}_\O(z,\O\setminus D) \k_{D}(z;v)\leq \k_\O(z;v) \: \: \: \: \: \: z\in D,\:\: v\in\Cn.
\end{equation}
As noted earlier, the first statement of the following proposition shows that the limit in \eqref{multiplicativelocalizationequation} is well-defined. Moreover, we will use its second statement in the proof of Theorem \ref{additivelocalizationresult}.

\begin{prop}\label{Wpointsandhyperbolicity} $\:$
	
(a) Suppose that $p\in\partial\O$ is a $w$-point. Then $\O$ is hyperbolic at $p\in\partial \O$.

(b) Suppose that $\O$ is hyperbolic at $p\in\partial\O$. Then $\O$ is hyperbolic near $p\in\partial\O$.
	\end{prop}

\begin{proof} (a) Since $k_\O(z,w)=(z|z)^\O_w$ by \eqref{Wpointcondition} we have $\lim_{z\rightarrow p, w\rightarrow q\neq p} k_\O(z,w)=\infty$. In particular, this shows that for any neighbourhood $U$ of $p$ and any $c\in\mathbb{R}^{+}$, there exists another neighbourhood $V\subset\subset U$ of $p$ such that
$$ k_\O(\O\cap V,\O\setminus U) > c .$$
As \eqref{localhyperbolicity} is equivalent to \eqref{strongerlocalhyperbolicity}, (a) follows.
	
(b) We will consider two cases. Let $V\subset\subset U$ be two neighbourhoods of $p$ such that $U$ is bounded and we have $k_\O(\O\cap V,\O\setminus U) =: c' > 0$. Let $z\neq w\in\O\cap V$ and $\gamma_n:I_n\rightarrow \O$ be $\frac{1}{n}$-geodesics joining $z$ to $w$. Suppose that each $\gamma_n$ remains in $\O\cap V$. Then as $\O\cap U$ is bounded, by construction \eqref{roydenslocalizationlemma} yields that
	$$ k_\O(z,w)+\frac{1}{n}\geq \lk_{\O}(\gamma) \geq \tanh(c')\lk_{\O\cap U}(\gamma)\geq \tanh(c') k_{\O\cap U}(z,w) > 0 .$$
	
	Letting $n\rightarrow \infty$, the claim follows. Now, we consider the other case. We assume that for large $n$, each $\gamma_n$ above leaves $\O\cap V$. Set $t_n:=\inf\{t\in I_n: \gamma_n(t)\in\partial V \cap \O\}$ and set $\gamma'_n=\gamma_n|_{I'_n}$, where $I'_n:=\{t\in I_n: t\leq t_n\}$.
	Then, by \eqref{roydenslocalizationlemma}, we obtain $$ k_\O(z,w)+\frac{1}{n}\geq \lk_{\O}(\gamma_n) \geq \lk_\O(\gamma'_n) \geq \tanh(c')\lk_{\O\cap U}(\gamma'_n)\geq \tanh(c')k_{\O\cap U}(z,\gamma_n(t_n)) .$$
	
	$\O\cap U$ is bounded so $\inf_{n\in\mathbb{N}}k_{\O\cap U}(z,\gamma_n(t_n)) =: c'' > 0$. Letting $n\rightarrow\infty$ we see that also in this case $k_\O(z,w)>0$. Hence, the proposition follows.
\end{proof}

\textit{Proof of Theorem \ref{multiplicativelocalizationresult}.}
We will first show the following:

\textit{Claim 1.} Under the hypothesis of the theorem, for any neighbourhood $U$ of $p$ and any $\epsilon>0$, there exists another neighbourhood $V$ of $p$ such that $$ \tilde l_\O (z,w) > 1-\epsilon, \:\:\:\:\: z\in \O\cap V, \: w \in \O\setminus U .$$

\textit{Proof.} This  follows directly from the first part of the proof of Proposition \ref{Wpointsandhyperbolicity}.

\textit{Claim 2. } Under the hypothesis of the theorem, for any neighbourhood $V$ of $p$,
there exists another neighbourhood $W\subset\subset V$ of $p$ and $\epsilon > 0$ such that for any $z,w\in\O\cap W$ if $\gamma:I\rightarrow\O$ is an $\epsilon$-geodesic joining $z$ to $w$ we have that $\gamma(I)\subset\O\cap V$.

\textit{Proof.} Let $V$ be any neighbourhood of $p$. Since $p$ is a $w$-point, by \eqref{Wpointcondition} we can find a $c > 0$ and a small enough neighbourhood $W\subset\subset V$ of $p$ such that
\begin{equation}\label{claim2lowerbound} k_\O(z,o)+k_\O(w,o) \geq k_\O(z,w)+c, \:\:\:\:\: z,w\in\O\cap W, \:\: o\in\O\setminus V .
\end{equation}

Taking $z,w\in \O\cap V$ this immediately proves Claim 2 for $\epsilon\leq c$.

We continue with the proof of Theorem \ref{multiplicativelocalizationresult}.
Let $U$ be any neighbourhood of $p$.

By Claim 1, for any $\epsilon>0$, we can find a neighbourhood $V$ of $p$ such that we have
$$ \tilde{l}_\O(\O\cap V,\O\setminus U) > 1-\epsilon.$$

Let $V$ be as above, by Claim 2, we can choose another neighbourhood $W\subset\subset V$ such that whenever $z,w\in\O\cap V$ and $\gamma_n:I_n\rightarrow \O$ are $\epsilon_n$-geodesics joining $z$ to $w$ with $\epsilon_n\rightarrow 0$ we must have $\gamma_n(I_n)\subset\O\cap V$ for large enough $n$. Furthermore, by Proposition \ref{Wpointsandhyperbolicity} we can take $W$ small enough so that whenever $z\neq w\in \O\cap W$ we have $k_\O(z,w)>0$. Thus, by applying \eqref{roydenslocalizationlemma} for large enough $n$ we get
$$ k_\O(z,w)+\epsilon_n\geq\lk_{\O}(\gamma_n)\geq (1-\epsilon)\lk_{\O\cap U}(\gamma_n)\geq (1-\epsilon)k_{\O\cap U}(z,w) .$$ Letting $\epsilon_n\rightarrow 0$ we obtain $$k_\O(z,w)\geq(1-\epsilon)k_{\O\cap U}(z,w) .$$ As $\epsilon>0$ and $z,w\in \O\cap W$ were arbitrary, Theorem \ref{multiplicativelocalizationresult} follows.\qed
\smallskip

To prove Theorem \ref{additivelocalizationresult}, we will follow the proof of \cite[Lemma 3.7]{S} (see also the proof of \cite[Theorem 10]{NOT}). By carefully keeping track of the estimates given in the aforementioned proofs, the $w$-point condition will allow us to provide a strong additive localization result.
\smallskip

\textit{Proof of Theorem \ref{additivelocalizationresult}.}
We fix a neighbourhood $U$ of $p$ such that $\O\setminus U\neq \emptyset$. (otherwise there is nothing to prove) Since $p$ is a $w$-point by \eqref{Wpointcondition} we can find neighbourhoods of $p$, say $V_{n}\subset\subset U$, such that
\begin{equation}\label{applicationofWlastproof}
	k_\O(z,o)+k_\O(w,o)-k_\O(z,w)>n,\quad z,w\in \O\cap V_n,\: o \in \O\setminus U.
\end{equation}

Notice that the conditions of Theorem \ref{additivelocalizationresult} is stronger than the conditions of Theorem \ref{multiplicativelocalizationresult}. So, the Claim 2 given in the proof of Theorem \ref{multiplicativelocalizationresult} remains true. Therefore by Claim 2, for each $V_n$ we can find another neighbourhood $W_n\subset\subset V_n$ of $p$ such that whenever $z,w\in \O\cap W_n$ and $\gamma:I\rightarrow\O$ is a $\epsilon_n$-geodesic joining $z$ to $w$ we have $\gamma(I)\subset \O\cap V_n$. Let $z_n,w_n\rightarrow p$ be arbitrary. By passing to a subsequence if necessary we assume that $z_n,w_n\in \O\cap W_n$. Choose a sequence of $\epsilon_n$-geodesics as above with $\epsilon_n\rightarrow 0$, say $\gamma_n:I_n\rightarrow \O$, joining $z_n$ to $w_n$. By construction, we can assume that $\gamma_n(I_n)\subset\O\cap V_n$.

As $p$ is a $w$-point Proposition \ref{Wpointsandhyperbolicity} shows that $\O$ is hyperbolic at $p$. Using the important estimate given in \cite[Lemma 3.1]{S} we get that $$ k_{\O\cap U}(z_n,w_n) \leq \lk_{\O\cap U}(\gamma_n) = \int_{I_n} \kappa_{\O\cap U}(\gamma(t);\gamma'(t))dt \leq \int_{I_n}(1+Ce^{-k_\O(\gamma_n(t),\O\setminus U)})\k_\O(\gamma_n(t);\gamma'_n(t)))dt  .$$

Since each $\gamma_n$ is an $\epsilon_n$-geodesic, by the above we get
$$
E_n:=k_{\O\cap U}(z_n,w_n)-k_\O(z_n,w_n)\leq C\int_{I_n} e^{-k_\O(\gamma_n(t),\O\setminus V)} \k_\O(\gamma_n(t);\gamma'_n(t))dt - \epsilon_n.
$$

We will show that $E_n\rightarrow 0$. To do so, we note that since $\O$ is hyperbolic at $p$ by Proposition \ref{Wpointsandhyperbolicity} it is hyperbolic near $p$. By passing to the limit if necessary, \cite[Lemma 6.10]{BNT} (see also \cite[Proposition 4.4]{BZ1}) allows us to parametrize each $\gamma_n$ so that
$$ k_\O(\gamma_n(t_1),\gamma_n (t_2))\leq |t_1-t_2| = \lk_{\O}(\gamma|_{[t_1,t_2]}) \leq k_\O(\gamma_n(t_1),\gamma_n(t_2))+\epsilon_n, \: \: \: \: \: t_1 \leq t_2 \in I_n $$ and $\k_\O(\gamma_n(t);\gamma'_n(t))=1$ almost everywhere on $I_n$.

With these parameterizations we set $$t_n:=\inf\{t\in I_n: \: \forall s\in I_n, \: k_\O(\gamma_n(t),\O\setminus U)\leq k_\O(\gamma_n(s),\O\setminus U) \}.$$

Let $o\in \O$ be arbitrary, since $p$ is a $v$-point we can apply the arguments given in the
proof of \cite[Theorem 10]{NOT} to get
$$k_\O(\gamma_n(t),\O\setminus V) \geq \frac{1}{2}\left(\k_\O(\gamma_n(t),o)+k_\O(\gamma_n(t_n),o)\right) - C' \geq \dfrac{1}{2}|t-t_n|+\frac{n}{2}- C'' $$ so
$$ E_n\leq C''' e ^{-\frac{n}{2}}\int_{I_n} e^{-\frac{|t-t_n|}{2}} dt - \epsilon_n.$$

Letting $n\rightarrow \infty$, $E_n\rightarrow 0$. As $\{z_n\}_{n\in\mathbb{N}},\{w_n\}_{n\in\mathbb{N}}$ was arbitrary, Theorem
\ref{additivelocalizationresult} follows.\qed

\begin{rem}
The conditions given in Theorem \ref{multiplicativelocalizationresult} and Theorem \ref{additivelocalizationresult} can be localized.
\end{rem}

To see that the $v$-point condition can be localized, suppose that $p\in\partial \O$ is a $w$-point of $\O$ and $v$-point of $\O\cap U$, where $U$ is a bounded neighbourhood of $p$. By Proposition \ref{Wpointsandhyperbolicity}, $\O$ is hyperbolic at $p$.
Therefore due to the discussion after its proof \cite[Theorem 9]{NOT} applies to the pair $k_\O,k_{\O\cap U}$. Then one may conclude that $p$ is indeed a $v$-point of $\O$. This observation extends Theorem \ref{additivelocalizationresult} to the aforementioned case.

On the other hand, it is not clear that being a local $w$-point implies hyperbolicity at $p$. Thus to use \cite[Theorem 9]{NOT} for localizing the $w$-point condition one needs to assume hyperbolicity of $\O$ at $p$.
\medskip

\noindent{\bf Acknowledgements.} The authors are grateful to Pascal J. Thomas for useful discussions. The authors also thank to the referee for valuable comments.

\end{document}